\documentclass[fleqn]{birkjour}
 \newtheorem{thm}{Theorem}[section]

 \newtheorem{prop}[thm]{Proposition}
 \theoremstyle{definition}
 
 \theoremstyle{remark}

 \numberwithin{equation}{section}

\newtheorem{result}[thm]{Result}

\newcommand{\co}{\nabla}

\begin{document}

%
%
%
%
%
%
%
%
%

\title[On tangent sphere bundles with  contact pseudo-metric structures]
 {On tangent sphere bundles with  contact pseudo-metric structures}

\author[N. Ghaffarzadeh]{Narges Ghaffarzadeh}

\address{Department of Pure Mathematics,\\ Faculty of Mathematical Sciences,\\ University of Tabriz, Tabriz, Iran.}

\email{n.ghaffarzadeh@tabrizu.ac.ir}

\author[M. Faghfouri]{Morteza Faghfouri}
\address{Department of Pure Mathematics,\\ Faculty of Mathematical Sciences,\\ University of Tabriz, Tabriz, Iran.}
\email{faghfouri@tabrizu.ac.ir}


\date{January 1, 2004}
\dedicatory{To my boss}

\begin{abstract}
In this paper, we introduce a contact pseudo-metric structure on a tangent sphere bundle $T_\varepsilon M$.  we prove that the tangent sphere bundle $T_{\varepsilon}M$ is $(\kappa, \mu)$-contact pseudo-metric manifold
 if and only if the manifold $M$ is of constant sectional curvature. Also, we prove that this structure on the tangent sphere bundle is $K$-contact iff the base manifold has constant curvature $\varepsilon$.
\end{abstract}

\maketitle
\maketitle
 \section{Introduction}
In 1956, S. Sasaki \cite{sasaki:bundle} introduced a Riemannian metric on tangent bundle $TM$ and tangent sphere bundle $T_1M$ over a Riemannian manifold $M$. Thereafter, that metric was called the Sasaki metric. In 1962, Dombrowski \cite{Dombrowski} also showed at each $Z\in TM,\, TM_Z=HTM_Z\oplus VTM_Z$, where $HTM_Z$ and $VTM_Z$ orthogonal subspaces of dimension $n$, called horizontal and vertical distributions, respectively. He  defined an almost K\"{a}hlerian structure on $TM$ and proved that it is K\"{a}hlerian manifold if $M$ is flat. In the same year, Tachibana and Okumura \cite{tachibana:bundlecomplex} showed  that the tangent bundle
space $TM$ of any non-flat Riemannian space $M$  always admits an almost K\"{a}hlerian structure, which is not K\"{a}hlerian.  Tashiro \cite{Tashiro:contactbundle} introduced a contact metric structure on the unit tangent sphere bundle $T_1M$ and he proved that  contact metric structure on $T_1M$ is $K$-contact iff $M$ has constant curvature 1, in which case the structure is Sasakian.

 Kowalski \cite{Kowalski} computed  the curvature tensor of Sasaki metric. Thus, on $T_1M,\, R(X,Y)\xi$ can be computed by the formulas for the curvature of $TM$.

 In \cite{Blair:Contactmetricmanifoldssatisfyingnullitycondition}, Blair et al. introduced $(\kappa,\mu)$-contact Riemannian manifolds and proved that,  the tangent sphere bundle $T_{1}M$ is a $(\kappa, \mu)$-contact Riemannian manifold iff the base manifold $M$ is of constant sectional curvature $c$.

 Takahashi \cite{Takahashi:SasakianManifoldWithPseud} introduced contact pseudo-metric structures $(\eta,g)$, where $\eta$ is a contact one-form and $g$ is a pseudo-Riemannian metric associated to it, are a natural generalization of contact metric
structures. Recently, contact pseudo-metric manifolds  have been studied by Calvaruso and Perrone \cite{Calvaruso.Perrone:ContactPseudoMtricManifolds,Perrone:CurvatureKcontact} and authors of this paper \cite{GhaffarzadehFaghfouri} introduce and study $(\kappa, \mu)$-contact pseudo-metric manifolds.

 In this paper, we suppose that $(M,g)$ is pseudo-metric manifold and  define pseudo-metric on $TM$. Also, we introduce contact pseudo-metric structures $(\varphi,\xi,\eta,g_{cm})$  on $T_\varepsilon M$ and prove that \begin{align*}
\bar{R}(X,Y)\xi=c(4\varepsilon-(c+2))\{\eta(Y)X-\eta(X)Y\}-2\varepsilon c\{\eta(Y)hX-\eta(X)hY\}
\end{align*}
 if and only if the base manifold $M$ is of constant sectional curvature. That is, the tangent sphere bundle $T_{\varepsilon}M$ is a $(\kappa, \mu)$-contact pseudo-metric manifold iff the base manifold $M$ is of constant sectional curvature. Also, the contact pseudo-metric structure $(\varphi,\xi,\eta,g_{cm})$ on $T_{\varepsilon}M$ is $K$-contact if and only if the base manifold $(M,g)$ has constant curvature $\varepsilon$.

\section{Preliminaries}
Let $(M,g)$ be a pseudo-metric manifold and $\co$ the associated Levi-Civita connection and  $R=[\co,\co]-\co_{[,]}$ the curvature tensor. The tangent bundle of $M$, denoted by $TM$,  consists of pairs $(x,u)$, where $x\in M$ and $u\in T_xM$,( i.e.  $TM=\cup_{x\in M}T_xM$). The mapping $\pi:TM\to M, \pi(x, u)=x$ is the natural projection.\\
The tangent space $T_{(x,u)}TM$ splits into  the vertical subspace $VTM_{(x,u)}$ is given by
$VTM_{(x,u)}:=\ker\pi_{*}\vert_{(x,u)}$
and the horizontal subspace $HTM_{(x,u)}$ with respect to $\co$:
$$T_{(x,u)}TM=VTM_{(x,u)}\oplus HTM_{(x,u)}.$$
For every $X\in T_{x}M$, there is a unique vector $X^{h}\in HTM_{(x,u)}$, such that $\pi_{*}(X^{h})=X$. It is called the horizontal lift of $X$ to $(x,u)$. Also,  there is a unique vector $X^{v}\in VTM_{(x,u)}$, such that $X^{v}(df)=Xf$ for all  $f\in C^\infty(M)$. $X^{v}$ is called the vertical lift of $X$ to $(x,u)$. The maps
$X\mapsto X^{h}$  between $T_{x}M$ and $HTM_{(x,u)}$, and   $X\mapsto X^{v}$ between $T_{x}M$ and $VTM_{(x,u)}$ are isomorphisms.  Hence, every tangent vector $\bar{Z}\in T_{(x,u)}TM$ can be decomposed $\bar{Z}=X^{h}+Y^{v}$ for uniquely determined vectors $X,Y\in T_{x}M$. The horizontal ( respectively, vertical) lift of a vector field $X$ on $M$ to $TM$ is the vector field $X^{h}$ (respectively, $X^{v}$ ) on $M$, whose value at the point $(x,u)$ is the horizontal (respectively, vertical) lift of $X_{x}$ to $(x,u)$.\\
A system of local coordinate $(x^{1},\ldots,x^{n})$ on an open subset $U$ of $M$ induces on $\pi^{-1}(U)$ of $TM$ a system of local coordinate $(\bar{x}^{1},\ldots,\bar{x}^{n};u^{1},\ldots,u^{n})$ as follows:
$\bar{x}^{i}(x,u)=(x^{i}\circ\pi)(x,u)=x^{i}(x), u^{i}(x,u)=dx^{i}(u)=ux^{i}$
for $i=1,\ldots,n$ and $(x,u)\in\pi^{-1}(U)$. With respect to the induced  local coordinate system, the horizontal and vertical lifts of a vector field
$X= X^{i}\frac{\partial}{\partial x^{i}}$
on $U$ are given by
\begin{gather}
X^{h}= (X^{i}\circ\pi)\frac{\partial}{\partial\bar{x}^{i}}-u^{b}((X^{a}\Gamma^{i}_{ab})\circ\pi)\frac{\partial}{\partial u^{i}},\label{0000}\\
X^{v}=(X^{i}\circ\pi)\frac{\partial}{\partial u^{i}},\label{0001}
\end{gather}
where $\Gamma^{i}_{jk}$ are the local components of $\co$, i.e., $\co_{\frac{\partial}{\partial x^{j}}}\frac{\partial}{\partial x^{k}}=\Gamma^{i}_{jk}\frac{\partial}{\partial x^{i}}$. The span of the horizontal lifts at $t\in TM$ is called the horizontal subspace
of $T_{t}TM$.
For all $t\in TM$, the connection map $\mathcal{K}:TTM\to TM$ is given by
$\mathcal{K}X_t^{h}=0 $ and $ \mathcal{K}X_{t}^{v}=X_{\pi(t)}$ \cite{Dombrowski}.

From (\ref{0000}) and (\ref{0001}), one can easily calculate the brackets of vertical and horizontal lifts:
\begin{gather}
[X^{h},Y^{h}]=[X,Y]^{h}-v\{R(X,Y)u\},\label{0002}\\
[X^{h},Y^{v}]=(\co_{X}Y)^{v},\label{0003}\\
[X^{v},Y^{v}]=0,\label{0004}
\end{gather}
for all $X,Y\in\Gamma(TM)$. We use some notation, due to M. Sekizawa (\cite{Sekizawa}). Let $T$ be a tensor field of type $(1,s)$ on $M$ and $X_{1},\ldots,X_{s-1}\in\Gamma(TM)$, the vertical vector field $v\{T(X_{1},\ldots,u,\ldots,X_{s-1})\}$ on $\pi^{-1}(U)$ is given by
$$v\{T(X_{1},\ldots,u,\ldots,X_{s-1})\}:= u^{a}(T(X_{1},\ldots,\frac{\partial}{\partial x^{a}},\ldots,X_{s-1}))^{v}.$$
Analogously, one defines  the horizontal vector field $h\{T(X_{1},\ldots,u,\ldots,X_{s-1})\}$ and $h\{T(X_{1},\ldots,u,\ldots,u,\ldots,X_{s-2})\}$ and the vertical vector field $v\{T(X_{1},\ldots$ $,u,\ldots,u,\ldots,X_{s-2})\}$.  Note that these vector fields do not depend on the choice of coordinates on $U$. Let $(M,g)$ be a pseudo-metric manifold. On the tangent bundle $TM$, we can define a  pseudo-metric $Tg$ to be
\begin{align}\label{0005}
Tg(X^{h},Y^{h})=Tg(X^{v},Y^{v})=g(X,Y)\circ\pi,\quad Tg(X^{h},Y^{v})=0
\end{align}
for all $X,Y\in\Gamma(TM)$. We call it Sasaki pseudo-metric.
  According \eqref{0005}, If $\{E_1,\ldots,E_n\}$ is an orthonormal frame field on $U$ then $\{E_1^v,\ldots,E_n^v,E_1^h,\ldots,E_n^h\}$ is an orthonormal frame field on $\pi^{-1}(U)$. So, we have the following:

  \begin{pro}
If the index of $g$ is $\nu$ then the index of the Sasaki pseudo-metric $Tg$ is $2\nu$.
   \end{pro}

  Let $\tilde{\co}$ be  the Levi-Civita connection of $Tg$. It is easy to check that for $X,Y\in\Gamma(TM)$ and $(x,u)\in TM$(see \cite{Kowalski} for more details):
\begin{align}\label{0006}
\begin{split}
(\tilde{\co}_{X^{v}} Y^{v})&=0,\qquad (\tilde{\co}_{X^{v}} Y^{h})=\frac{1}{2}h\{R(u,X)Y\},\\
(\tilde{\co}_{X^{h}} Y^{v})&=({\co}_{X} Y)^{v}+\frac{1}{2}h\{R(u,Y)X\},\\
(\tilde{\co}_{X^{h}} Y^{h})&=({\co}_{X} Y)^{h}-\frac{1}{2}v\{R(X,Y)u\}.
\end{split}
\end{align}
\section{The curvature of the unit tangent sphere bundle with pseudo-metric}
Let $(TM,Tg)$ be the tangent bundle of $(M,g)$ endowed with its Sasaki pseudo-metric. We consider the hypersurface $T_{\varepsilon}M=\{(x,u)\in TM| g_{x}(u,u)=\varepsilon\}$, which we call  the unit tangent sphere bundle.  A unit normal vector field $N$ on $T_{x}M$ is the (vertical) vector field $N= u^{i}\frac{\partial}{\partial u^{i}}=u^{i}(\frac{\partial}{\partial x^{i}})^{v}$. $N$ is independent of the choice of local coordinates and it is defined globally on $TM$.
We introduce some more notation. If $X\in T_{x}M$, we define the tangential lift of $X$ to $(x,u)\in T_{\varepsilon}M$ by
\begin{align}\label{0007}
X^{t}_{(x,u)}=X^{v}_{(x,u)}-\varepsilon g(X,u)N_{(x,u)}.
\end{align}
Clearly, the tangent space to $T_{\varepsilon}M$ at $(x,u)$ is then spanned by vectors of the form $X^{h}$ and $X^{t}$, where $X\in T_{x}M$. Note that $u^{t}_{(x,u)}=0$. The tangential lift of a vector field $X$ on $M$ to $T_{\varepsilon}M$ is the vertical vector field $X^{t}$ on $T_{\varepsilon}M$, whose value at the point $(x,u)\in T_{\varepsilon}M$ is the tangential lift of $X_{x}$ to $(x,u)$. For a tensor field $T$ of type $(1,s)$ on $M$ and $X_{1},\ldots,X_{s-1}\in\Gamma(TM)$, we define the vertical vector fields $t\{T(X_{1},\ldots,u,\ldots,X_{s-1})\}$
and $t\{T(X_{1},\ldots,u,\ldots,u,\ldots,X_{s-2})\}$ on $T_{\varepsilon}M$ in the obvious way.\\
In what follows, we will give explicit expressions for the brackets of vector fields on $T_{\varepsilon}M$ involving tangential lifts, the Levi-Civita connection and the associated curvature tensor of the induced metric $\bar{g}$ on $T_{\varepsilon}M$.\\
First, for the brackets of vector fields on $T_{\varepsilon}M$ involving tangential lifts, we obtain
\begin{gather}
[X^{h},Y^{t}]=(\co_{X}Y)^{t},\label{0008}\\
[X^{t},Y^{t}]=\varepsilon g(X,u)Y^{t}-\varepsilon g(Y,u)X^{t}.\label{0009}
\end{gather}
Next, we denote by $\bar{g}$  the pseudo-metric induced on $T_{\varepsilon}M$ from  $Tg$ on $TM$.
\begin{pro}
The Levi-Civita connection $\bar{\co}$ of $(T_{\varepsilon}M,\bar{g})$ is described completely by
\begin{align}\label{0010}
\begin{split}
\bar{\co}_{X^{t}} Y^{t}&=-\varepsilon g(Y,u)X^{t},\\
\bar{\co}_{X^{t}} Y^{h}&=\frac{1}{2}h\{R(u,X)Y\},\\
\bar{\co}_{X^{h}} Y^{t}&=({\co}_{X} Y)^{t}+\frac{1}{2}h\{R(u,Y)X\},\\
\bar{\co}_{X^{h}} Y^{h}&=({\co}_{X} Y)^{h}-\frac{1}{2}t\{R(X,Y)u\}
\end{split}
\end{align}
for all vector fields $X$ and $Y$ on $M$.
\end{pro}
\begin{proof}
This is obtained by an easy calculation using \eqref{0006} and the following equation
$$\bar{\co}_{\bar{A}}\bar{B}=\tilde{\co}_{\bar{A}}\bar{B}-\varepsilon Tg(\bar{\co}_{\bar{A}}\bar{B},N)N,$$
for vector fields $\bar{A},\bar{B}$ tangent to $T_{\varepsilon}M$.
\end{proof}
\begin{pro}
The curvature tensor $\bar{R}$ of $(T_{\varepsilon}M,\bar{g})$ is described completely by
\begin{gather}
\bar{R}(X^{t},Y^{t})Z^{t}=\varepsilon\{-\bar{g}(X^{t},Z^{t})Y^{t}+\bar{g}(Z^{t},Y^{t})X^{t}\},\label{0011}\\
\bar{R}(X^{t},Y^{t})Z^{h}=(R(X,Y)Z)^{h}-\varepsilon\{g(Y,u)h(R(X,u)Z)+g(X,u)h(R(u,Y)Z)\}\nonumber\\
\quad\quad\quad\quad\quad\quad\quad+\frac{1}{4}h\{[R(u,X),R(u,Y)]Z\},\label{0012}\\
\bar{R}(X^{h},Y^{t})Z^{t}=-\frac{1}{2}(R(Y,Z)X)^{h}+\frac{\varepsilon}{2}\{g(Y,u)h(R(u,Z)X)+g(Z,u)h(R(Y,u)X)\}\nonumber\\
\quad\quad\quad\quad\quad\quad\quad-\frac{1}{4}h\{R(u,Y)R(u,Z)X\},\label{0013}\\
\bar{R}(X^{h},Y^{t})Z^{h}=\frac{1}{2}(R(X,Z)Y)^{t}-\frac{\varepsilon}{2}g(Y,u)t\{R(X,Z)u\}\nonumber\\
\quad\quad\quad\quad\quad\quad\quad-\frac{1}{4}t\{R(X,R(u,Y)Z)u\}+\frac{1}{2}h\{(\co_{X}R)(u,Y)Z\},\label{0014}\\
\bar{R}(X^{h},Y^{h})Z^{t}=(R(X,Y)Z)^{t}-\varepsilon g(Z,u)t\{R(X,Y)u\}\nonumber\\
\quad\quad\quad\quad\quad\quad\quad+\frac{1}{4}t\{R(Y,R(u,Z)X)u-R(X,R(u,Z)Y)u\}\nonumber\\
\quad\quad\quad\quad\quad\quad\quad+\frac{1}{2}h\{(\co_{X}R)(u,Z)Y-(\co_{Y}R)(u,Z)X\},\label{0015}\\
\bar{R}(X^{h},Y^{h})Z^{h}=(R(X,Y)Z)^{h}+\frac{1}{2}h\{R(u,R(X,Y)u)Z\}\nonumber\\
\quad\quad\quad\quad\quad\quad\quad-\frac{1}{4}h\{R(u,R(Y,Z)u)X-R(u,R(X,Z)u)Y\}\nonumber\\
\quad\quad\quad\quad\quad\quad\quad+\frac{1}{2}t\{(\co_{Z}R)(X,Y)u\}\label{0016}
\end{gather}
for all vector fields $X,Y$ and $Z$ on $M$.
\end{pro}
\begin{proof}
The proof is made by using the following equation and  equation (\ref{0010}) for the covariant derivative, (\ref{0002}), (\ref{0008}) and (\ref{0009}) for the brackets are explicitly calculated.
$$\bar{R}(\bar{A},\bar{B})\bar{C}=\bar{\co}_{\bar{A}}\bar{\co}_{\bar{B}}\bar{C}-\bar{\co}_{\bar{B}}\bar{\co}_{\bar{A}}\bar{C}-\bar{\co}_{[\bar{A},\bar{B}]}\bar{C}.$$
\end{proof}
\section{The contact pseudo-metric structure of the unit tangent sphere bundle}
First, we give some basic facts on contact pseudo-metric structures. A pseudo-Riemannian manifold $(M^{2n+1},g)$ has a contact pseudo-metric structure if it admits a vector field $\xi$, a one-form $\eta$ and a $(1,1)$-tensor field $\varphi$ satisfying
\begin{align}\label{0017}
\begin{split}
&\eta (\xi)= 1,\\
&\varphi^2(X)=-X+\eta(X)\xi,\\
&g(\varphi X,\varphi Y)=g(X,Y)-\varepsilon\eta(X)\eta(Y),\\
&d\eta(X,Y) =g(X,\varphi Y),
\end{split}
\end{align}
where $\varepsilon=\pm1$ and  $X,Y\in\Gamma(TM)$.
In this case, $(M,\varphi,\xi,\eta,g)$ is called a contact pseudo-metric manifold. In particular, the above conditions imply that the characteristic curves, i.e., the integral curves of the characteristic vector field $\xi$, are geodesics.\\
If $\xi$ is in addition a Killing vector field with respect to $g$, then the manifold is said to be a $K$-contact (pseudo-metric) manifold. Another characterizing property of such contact pseudo-metric manifolds is the following:\\
 geodesics which are orthogonal to $\xi$ at one point, always remain orthogonal to $\xi$. This yields a second special class of geodesics, the $\varphi$-geodesics.\\
Next, if $(M^{2n+1},\varphi,\xi,\eta,g)$ is a contact pseudo-metric manifold satisfying the additional condition
$N_{\varphi}(X,Y)+2d\eta(X,Y)\xi=0$ is said to be Sasakian,
where $N_{\varphi}(X,Y)=\varphi^2[X,Y]+[\varphi X,\varphi Y]-\varphi[\varphi X,Y]-\varphi[X,\varphi Y]$ is the Nijenhuis torsion tensor of $\varphi$. Moreover, an almost contact pseudo-metric manifold $(M^{2n+1},\varphi,\xi,\eta,g)$ is a Sasakian manifold if and only if $(\co_X\varphi)Y=g(X,Y)-\varepsilon\eta(Y)X$.
 In particular, one can show that the characteristic vector field $\xi$ is a Killing vector field. Hence, a Sasakian manifold is also a $K$-contact manifold(see \cite{Calvaruso.Perrone:ContactPseudoMtricManifolds,Perrone:CurvatureKcontact} for more details). If a contact pseudo-metric manifold satisfying $R(X,Y)\xi= \varepsilon\kappa(\eta(Y )X -\eta(X)Y ) +\varepsilon \mu(\eta(Y )hX -\eta(X)hY )$, we call  $(\kappa, \mu)$-contact pseudo-metric manifold, where $(\kappa ,\mu)\in \mathbb{R}^2$. $(\kappa, \mu)$-contact pseudo-metric manifold is Sasakian if and only if $\kappa=\varepsilon$(see \cite{GhaffarzadehFaghfouri} for more details).

Take now an arbitrary pseudo-metric manifold $(M,g)$. One can easily define an almost complex structure $J$ on $TM$ in the following way:
\begin{align}\label{0018}
JX^{h}=X^{v},\quad JX^{v}=-X^{h}
\end{align}
for all vector fields $X$ on $M$.
 From (\ref{0002}), (\ref{0003}) and (\ref{0004}), we have the almost complex structure $J$ is integrable if and only if $(M,g)$ is flat.
 From the definition (\ref{0005}) of the pseudo-metric $Tg$ on $TM$, it follows immediately that $(TM,Tg,J)$ is almost Hermitian. Moreover, $J$ defines an almost K\"{a}hlerian structure. It is a K\"{a}hler manifold only when $(M,g)$ is flat\cite{Dombrowski}.\\
Next, we consider the unit tangent sphere bundle $(T_{\varepsilon}M,\bar{g})$, which is isometrically embedded as a hypersurface in $(TM,Tg)$ with unit normal field $N$. Using the almost complex structure $J$ on $TM$, we define a unit vector field $\xi'$, a one-form $\eta'$ and a $(1,1)$-tensor field $\varphi'$ on $T_{\varepsilon}M$ by
\begin{align}\label{0019}
\xi'=-JN,\quad JX=\varphi' X+\eta'(X)N.
\end{align}
In local coordinates, $\xi'$, $\eta'$ and $\varphi'$ are described by
\begin{align}\label{0020}
\begin{split}
&\xi'=u^{i}(\frac{\partial}{\partial x^{i}})^{h},\\
&\eta'(X^{t})=0,\quad\eta'(X^{h})=\varepsilon g(X,u),\\
&\varphi'(X^{t})=-X^{h}+\varepsilon g(X,u)\xi',\quad\varphi'(X^{h})=X^{t},
\end{split}
\end{align} where  $X,Y\in\Gamma(TM)$.
It is easily checked that these tensors satisfy the conditions (\ref{0017}) excepts or the last one, where we find $\bar{g}(X,\varphi' Y)=2 d\eta'(X,Y)$. So strictly speaking, $(\varphi',\xi',\eta',\bar{g})$ is not a contact pseudo-metric structure. Of course, the difficulty is easily rectified and
$$\eta=\frac{1}{2}\eta',\quad\xi=2\xi',\quad\varphi=\varphi',\quad g_{cm}=\frac{1}{4}\bar{g}$$
is taken as the standard contact pseudo-metric structure on $T_{\varepsilon}M$. In local coordinates,
with respect to induce the  local coordinates $(x^{i},u^{i})$ on $TM$, the characteristic vector field is given by
$$\xi=2u^{i}(\frac{\partial}{\partial x^{i}})^{h},$$
the vector field $u^{i}(\frac{\partial}{\partial x^{i}})^{h}$ is the well-known geodesic flow on $T_{\varepsilon}M$. Before beginning our theorems, we explicitly obtain  the covariant derivatives of $\xi$ and $\varphi$. For a horizontal tangent vector field, we may  use a horizontal lift again. Then
$$\bar{\co}_{X^{h}}\xi=\tilde{\co}_{X^{h}}\xi=-v\{R(X,u)u\}$$
and hence for any horizontal vector $X^{h}$ at $(x,u)\in T_{\varepsilon}M$, we have
$$\bar{\co}_{X^{h}}\xi=-v\{R(X,u)u\}=-t\{R(X,u)u\}.$$
For a vertical vector field $X^{v}$ tangent to $T_{\varepsilon}M$, we have
$$\bar{\co}_{X^{v}}\xi=\tilde{\co}_{X^{v}}\xi=-2\varphi X^{v}-h\{R(X,u)u\}.$$
Since $J(\frac{\partial}{\partial x^{i}})^{h}=(\frac{\partial}{\partial x^{i}})^{v}$, or in terms of tangential lifts of a vector $X$ on $M$,
$$\bar{\co}_{X^{t}}\xi=-2\varphi X^{t}-h\{R(X,u)u\}.$$
Comparing with $\bar{\co}_{X}\xi=-\varepsilon\varphi X-\varphi hX$  on $T_{\varepsilon}M$ for a vertical vector $V$ and a horizontal vector $X$ orthogonal to $\xi$, $hV$ and $hX$ are given by
\begin{align}\label{085}
hV=(2-\varepsilon)V-v\{R(\mathcal{K}V,u)u\}\quad\text{and}\quad hX=-\varepsilon X+h\{R(\pi_{*}X,u)u\}.
\end{align}
Also for any tangent vector fields $X$ and $Y$, we have
\begin{align*}
(\bar{\co}_{X}\varphi)Y=&\tilde{\co}_{X}JY-(\bar{\co}_{X}\eta')(Y)N+\eta'(Y)AX\\
&-\varepsilon\bar{g}(X,A\varphi Y)N-J(\tilde{\co}_{X}Y)-\varepsilon\bar{g}(X,AY)\xi'.
\end{align*}
We present two computations, one done in each manner.\\
As before, for $X,Y$ horizontal vector fields, we suppose that they are horizontal lifts, and we have
$$(\bar{\co}_{X^{h}}\varphi)Y^{h}=\frac{1}{2}h\{R(u,X)Y\},$$
where we  used the first Bianchi identity.\\
For $Y^{v}=Y^{i}\frac{\partial}{\partial u^{i}}$ a vertical vector field tangent to $T_{\varepsilon}M$ and $X^{h}$ a horizontal tangent vector, we have
$$(\bar{\co}_{X^{h}}\varphi)Y^{v}=\frac{1}{2}t\{R(X,u)Y\},$$
where we  used
$$(\bar{\co}_{X^{h}}\eta')(Y^{v})=\varepsilon\bar{g}(Y^{v},\bar{\co}_{X^{h}}\xi')=-\frac{\varepsilon}{2}g(Y,R(X,u)u)=\frac{\varepsilon}{2}Tg(N,(R(X,u)Y)^{v}).$$
Similarly, we obtain
$$(\bar{\co}_{X^{v}}\varphi)Y^{h}=t\{R(X,u)Y\}-2\eta(Y^{h})X^{v},$$
$$(\bar{\co}_{X^{v}}\varphi)Y^{v}=\frac{1}{2}h\{R(X,u)Y\}+2\varepsilon g_{cm}(X,Y)\xi.$$
\begin{theorem}\label{th:R}
Let $(\varphi,\xi,\eta,g_{cm})$ be a contact pseudo-metric structure on the tangent sphere bundle $T_{\varepsilon}M$. Then
\begin{align}\label{R}
\bar{R}(X,Y)\xi=c(4\varepsilon-(c+2))\{\eta(Y)X-\eta(X)Y\}-2\varepsilon c\{\eta(Y)hX-\eta(X)hY\}
\end{align}
 if and only if the base manifold $M$ is of constant sectional curvature $c$.
\end{theorem}
\begin{proof}
Assume that the manifold $M$ is a pseudo-metric manifold of constant curvature $c$. Then from equations  (\ref{0011}-\ref{0016}), for $X,Y$ orthogonal to $\xi$, we have  $\bar{R}(X, Y)\xi=0$  and for a vertical vector $V$, we get $\bar{R}(V,\xi)\xi=c^{2}V$ and also, for a horizontal vector $X$ orthogonal to $\xi$, we obtain $\bar{R}(X,\xi)\xi=(4\varepsilon c- 3c^{2})X$. Moreover, from equations (\ref{085}), $hV=(2-\varepsilon (1+c))V$ and $hX= \varepsilon (c- 1)X$. Thus for all $X,Y$ on $T_{\varepsilon}M$, the curvature tensor on $T_{\varepsilon}M$ satisfies
$$\bar{R}(X,Y)\xi=c(4\varepsilon-(c+2))\{\eta(Y)X-\eta(X)Y\}-2\varepsilon c\{\eta(Y)hX-\eta(X)hY\}.$$
Conversely, if the contact pseudo-metric structure on $T_{\varepsilon}M$ satisfies the condition 
$$\bar{R}(X,Y)\xi=\varepsilon\kappa\{\eta(Y)X-\eta(X)Y\}+\varepsilon \mu\{\eta(Y)hX-\eta(X)hY\},$$
where $\kappa=c(4-\varepsilon(c+2)), \mu= -2c $ then
\begin{align}\label{086}
\bar{R}(X,\xi)\xi=\varepsilon\kappa X+\varepsilon\mu hX,
\end{align}
for any $X$ orthogonal to $\xi$. Now, for a vector $u$ on $M$, that $g(u,u)=\varepsilon$ define a symmetric operator
$\psi_{u}:\langle u\rangle^{\perp}\to \langle u\rangle^{\perp}$
 by $\psi_{u}X= R(X, u)u$. By placing the equation (\ref{085}) in (\ref{086}), we get
 \begin{align}\label{1}
 \bar{R}(V,\xi)\xi=\varepsilon(\kappa+\mu(2-\varepsilon))V-\varepsilon\mu\, v\{\psi_{u}\mathcal{K}V\}.
 \end{align}
 Also using equations (\ref{0011}-\ref{0016}), we have
 \begin{align}\label{2}
 \bar{R}(V,\xi)\xi=-v(R(R(u,\mathcal{K}V)u,u)u)=v\{\psi^{2}_{u}\mathcal{K}V\}.
  \end{align}
From a comparison of equations \eqref{1} and \eqref{2},  the operator $\psi_{u}$ satisfies the equation
\begin{align}\label{1.1}\psi^{2}_{u}+\varepsilon\mu \psi_{u}-\varepsilon(\kappa+(2-\varepsilon)\mu)I=0.\end{align}
In a similar way, for a horizontal $X$ orthogonal to $\xi$,
$$\bar{R}(X,\xi)\xi=(\varepsilon\kappa-\mu)X+\varepsilon\mu\, h(\psi_{u}\pi_{*}X),$$
and, from equations  \eqref{0011}-\eqref{0016}, we get
$$\bar{R}(X,\xi)\xi=h(4\psi_{u}\pi_{*}X-3\psi^{2}_{u}\pi_{*}X),$$
giving
\begin{align}\label{1.2}3\psi^{2}_{u}+(\varepsilon\mu-4)\psi_{u}+(\varepsilon\kappa-\mu)I=0.\end{align}
From  equations  \eqref{1.1} and \eqref{1.2}, the eigenvalues $a$ of $\psi_{u}$ satisfy the two quadratic equations
$$a^{2}+\varepsilon\mu a-(\varepsilon\kappa+(2\varepsilon-1)\mu)=0,\quad a^{2}+\frac{\varepsilon\mu-4}{3}a+\frac{\varepsilon\kappa-\mu}{3}=0.$$
Thus, the common root of the above quadratic equations is $a=\varepsilon c$.
Hence from $\psi_{u}X= R(X, u)u=\varepsilon cX$ and $g(u,u)=\varepsilon$, $M$ is of constant curvature $c$.
\end{proof}
We can now rephrase Theorem \ref{th:R} as the following result.
\begin{result}\label{4.2}
 The tangent sphere bundle $T_{\varepsilon}M$ is $(\kappa, \mu)$-contact pseudo-metric manifold
 if and only if the base manifold $M$ is of constant sectional curvature $c$ and $\kappa=c(4-\varepsilon(c+2)), \mu= -2c $.
\end{result}
We now have the following theorem about the $K$-contact structure.
\begin{theorem}
The contact pseudo-metric structure $(\varphi,\xi,\eta,g_{cm})$ on $T_{\varepsilon}M$ is $K$-contact if and only if the base manifold $(M,g)$ has constant curvature $\varepsilon$.
\end{theorem}
\begin{proof}
Using Theorem 3.3 of \cite{Perrone:CurvatureKcontact}, since sectional curvature of all nondegenerate plane sections containing $\xi$ equals $\varepsilon$, therefore  $T_{\varepsilon}M$ is $K$-contact and conversely.
\end{proof}
\begin{theorem}
Let $M$ be an $n$-dimensional pseudo-metric manifold, $n>2$, of constant sectional curvature $c$. The tangent sphere bundle $T_{\varepsilon}M$ has constant $\varphi$-sectional curvature $(4c(\varepsilon-1)+\varepsilon c^{2})$ if and only if $c=2\varepsilon\pm\sqrt{5}$.
\end{theorem}
\begin{proof}
Using Theorem 3.5 of \cite{GhaffarzadehFaghfouri}, $(\kappa, \mu)$-contact pseudo-metric manifold $M$ has constant $\varphi$-sectional curvature  $-(\kappa+\varepsilon\mu)$ if and
only if $\mu=\varepsilon\kappa+1$. So, using Result \ref{4.2} and a straightforward calculation, the proof is completed.
\end{proof}
\begin{result}\label{4.3}
 For $\varepsilon=+1$, the tangent sphere bundle $T_{1}M$ is $(\kappa, \mu)$-contact pseudo-metric manifold
 if and only if the base manifold $M$ is of constant sectional curvature $c$ and $\kappa=c(2-c), \mu= -2c $. Also, $T_{1}M$ is Sasakian if and only if $c=1$.
\end{result}
\begin{result}\label{4.4}
 For $\varepsilon=-1$, the tangent sphere bundle $T_{-1}M$ is $(\kappa, \mu)$-contact pseudo-metric manifold
 if and only if the base manifold $M$ is of constant sectional curvature $c$ and $\kappa=c(6+c), \mu= -2c $. Also, $T_{-1}M$ is Sasakian if and only if $c=-3\pm2\sqrt{2}.$
\end{result}



\bibliographystyle{siam}



\begin{thebibliography}{10}

\bibitem{Blair:Contactmetricmanifoldssatisfyingnullitycondition}
{\sc D.~E. Blair, T.~Koufogiorgos, and B.~J. Papantoniou}, {\em Contact metric
  manifolds satisfying a nullity condition}, Israel J. Math., 91 (1995),
  pp.~189--214.

\bibitem{Calvaruso.Perrone:ContactPseudoMtricManifolds}
{\sc G.~Calvaruso and D.~Perrone}, {\em Contact pseudo-metric manifolds},
  Differential Geometry and its Applications, 28 (2010), pp.~615--634.

\bibitem{Dombrowski}
{\sc P.~Dombrowski}, {\em On the geometry of the tangent bundle}, J. Reine
  Angew. Math., 210 (1962), pp.~73--88.

\bibitem{GhaffarzadehFaghfouri}
{\sc N.~Ghaffarzadeh and M.~Faghfouri}, {\em On contact pseudo-metric manifolds
  satisfying a nullity condition}, Journal of Mathematical Analysis and
  Applications,  (2020), p.~124849.

\bibitem{Kowalski}
{\sc O.~Kowalski}, {\em Curvature of the induced {R}iemannian metric on the
  tangent bundle of a {R}iemannian manifold}, J. Reine Angew. Math., 250
  (1971), pp.~124--129.

\bibitem{Perrone:CurvatureKcontact}
{\sc D.~Perrone}, {\em Curvature of {$K$}-contact semi-{R}iemannian manifolds},
  Canad. Math. Bull., 57 (2014), pp.~401--412.

\bibitem{sasaki:bundle}
{\sc S.~Sasaki}, {\em On the differential geometry of tangent bundles of
  {R}iemannian manifolds}, Tohoku Math. J. (2), 10 (1958), pp.~338--354.

\bibitem{Sekizawa}
{\sc M.~Sekizawa}, {\em Curvatures of tangent bundles with {C}heeger-{G}romoll
  metric}, Tokyo J. Math., 14 (1991), pp.~407--417.

\bibitem{tachibana:bundlecomplex}
{\sc S.-i. Tachibana and M.~Okumura}, {\em On the almost-complex structure of
  tangent bundles of {R}iemannian spaces}, Tohoku Math. J. (2), 14 (1962),
  pp.~156--161.

\bibitem{Takahashi:SasakianManifoldWithPseud}
{\sc T.~Takahashi}, {\em Sasakian manifold with pseudo-{R}iemannian metric},
  T\^{o}hoku Math. J. (2), 21 (1969), pp.~271--290.

\bibitem{Tashiro:contactbundle}
{\sc Y.~Tashiro}, {\em On contact structures of tangent sphere bundles}, Tohoku
  Math. J. (2), 21 (1969), pp.~117--143.

\end{thebibliography}

\end{document}